\documentclass[a4paper,12pt]{amsart}
\usepackage{amssymb, amsmath}
\usepackage{amscd}
\numberwithin{equation}{section}
\usepackage{epsfig}
\usepackage{amsmath}
\usepackage{amsfonts,amssymb,amsopn}
\usepackage[all]{xy}
\SilentMatrices

\newcommand{\PP}{\mathbb{P}}
\newcommand{\C}{\mathbb{C}}
\newcommand{\zz}{\mathbb{Z}}

\newcommand{\Hom}{\mathrm{Hom}}
\newcommand{\End}{\mathrm{End}}
\newcommand{\Ker}{\mathrm{Ker}}
\newcommand{\Image}{\mathrm{Im}}
\newcommand{\Nm}{\mathrm{Nm}}

\newcommand{\rank}{\mathrm{rk}\,}
\newcommand{\Sym}{\mathrm{Sym}}

\newcommand{\Ox}{O_{X}}
\newcommand{\Kx}{K_{X}}

\newcommand{\Sec}{\mathrm{Sec}}

\newcommand{\Supp}{\mathrm{Supp}}
\newcommand{\mult}{\mathrm{mult}}

\newcommand{\Pic}{\mathrm{Pic}}
\def\map#1{\ \smash{\mathop{\longrightarrow}\limits^{#1}}\ }
\newcommand{\lra}{\longrightarrow}

\newcommand{\cD}{\mathcal{D}}
\newcommand{\Id}{\mathrm{Id}}
\newcommand{\Gr}{\mathrm{Gr}}

\newtheorem{theorem}{{\textbf Theorem}}[section]
\newtheorem{proposition}[theorem]{{\textbf Proposition}}
\newtheorem{corollary}[theorem]{{\textbf Corollary}}
\newtheorem{lemma}[theorem]{{\textbf Lemma}}
\newtheorem{criterion}[theorem]{{\textbf Criterion}}
\newtheorem{remit}[theorem]{{\textbf Remark}}
\newenvironment{remark}{\begin{remit}\rm}{\end{remit}}

\title[Theta divisors of stable bundles may be nonreduced]{Theta divisors of stable vector bundles may be nonreduced}

\author{George H.\ Hitching\\With an appendix by Christian Pauly}
\address{H\o gskolen i Oslo og Akershus, Postboks 4, 0130 Oslo, Norway. Tel.: +47 22 45 21 46}
\email{george.hitching@hioa.no}

\address{Laboratoire J.\ A. Dieudonn\'e, UMR CNRS 7351, Universit\'e de Nice Sophia-Antipolis, Parc Valrose, 06108 Nice Cedex 02, France. Tel.: +33 (0)4 9207 6202}
\email{pauly@math.unice.fr}



\begin{document}

\begin{abstract}
A generic strictly semistable bundle of degree zero over a curve $X$ has a reducible theta divisor, given by the sum of the theta divisors of the stable summands of the associated graded bundle. The converse is not true: Beauville and Raynaud have each constructed stable bundles with reducible theta divisors. For $X$ of genus $g \geq 5$, we construct stable vector bundles over $X$ of rank $r$ for all $r \geq 5$ with reducible and nonreduced theta divisors. We also adapt the construction to symplectic bundles.

In the appendix, Raynaud's original example of a stable rank 2 vector bundle with reducible theta divisor over a bi-elliptic curve of genus 3 is generalized to bi-elliptic curves of genus $g \geq 3$.
\end{abstract}

\maketitle

\section{Introduction}

Let $X$ be a complex projective smooth curve of genus $g \geq 2$. We write $J$ for the Jacobian variety parametrizing line bundles of degree $g-1$ over $X$. To a vector bundle $V \to X$ of degree zero we may associate the set
\begin{equation} \{ N \in J: h^{0}(N \otimes V) > 0 \}. \label{thetalocus} \end{equation}
If $V$ is a generic semistable bundle of rank $n$, then this is the support of a divisor $\Theta_V$ on $J$, called the \textsl{theta divisor of $V$}, which is algebraically equivalent to $n\Theta$. 
See Beauville \cite{B1,B} for details. (For certain nongeneric semistable or stable $V$, then (\ref{thetalocus}) is the whole of $J$. This phenomenon was first studied by Raynaud \cite{Ray}, and has subsequently attracted a good deal of attention.) 

Laszlo has given an analogue of the Riemann singularity theorem for $\Theta_V$:
\begin{theorem} If $\Theta_V$ is defined, then $\mult_N \Theta_V \geq h^{0}(N \otimes V)$. \label{Riemann} \end{theorem}
\begin{proof} This follows easily from \cite[Proposition V.2]{Las}. Laszlo's statement is for rank 2 bundles of slope $g-1$, but the arguments are easily adapted to bundles of slope zero, and apply in arbitrary rank if one assumes that $\Theta_V$ is defined. \end{proof}


If $V$ is strictly semistable, S-equivalent to a decomposable bundle $\bigoplus_i V_i$ where each $V_i$ is stable of degree zero, then $\Theta_V$ has the reducible theta divisor $\sum_i \Theta_{V_i}$, when this exists. The converse is not true; the first counterexample was given by Raynaud, who constructed a stable rank two bundle over a bi-elliptic curve of genus $3$ with a reducible theta divisor. This work was never published; in the appendix by Christian Pauly to the present article, Raynaud's construction is described and generalized to bi-elliptic curves of genus $g \geq 3$.

In \cite{B}, Beauville constructed stable bundles with reducible theta divisors over a general curve $X$ of genus $g \geq 3$, of ranks $\binom{g}{p}$, for $1 \leq p \leq g-1$. These are of the form $\bigwedge^p E_L$, where $E_L$ is the \textsl{evaluation bundle} defined by the exact sequence
\[ 0 \to E_L^* \to \Ox \otimes H^0 (L) \to L \to 0 \]
where $L$ is a general, very ample line bundle of degree $2g$. (Note that $E_L$ has slope $2$, so the theta divisor of $\bigwedge^p E_L$ belongs to $J^{g-3}$ instead of $J = J^{g-1}$.)

In the present work, we study a related phenomenon. If $V$ is a semistable bundle with theta divisor $\Theta_V$, then by Theorem \ref{Riemann}, the polystable bundle $V^{\oplus n}$ has the nonreduced theta divisor $n \Theta_V$. In light of Beauville's and Raynaud's constructions, it seems reasonable to expect that there also exist stable bundles with nonreduced theta divisors. In \S 3, we construct stable bundles with reducible and nonreduced theta divisors. 
Precisely:

\begin{theorem} Suppose $X$ has genus $g \geq 4$. Let $\Theta_V$ be the theta divisor of a generic stable bundle $V$ of rank $n \geq 1$ over $X$. Let $t$ be a positive integer with $2 \leq t < n(g-1)$. Then for any rank $r \geq tn + 2$, there exist stable bundles $W$ of degree zero and rank $r$ such that $\Theta_W$ exists and contains $(t-1) \Theta_V$ as a subscheme. \end{theorem}

\noindent In particular, letting $t = 2$ or $3$ and $n = 1$, we obtain:

\begin{corollary} Over any curve of genus $g \geq 5$, there exist stable bundles of all ranks $r \geq 4$ (resp., $r \geq 5$) with reducible (resp., reducible and nonreduced) theta divisor. \end{corollary}

These $W$ are obtained as extensions $0 \to E \to W \to M \to 0$ where $E$ is a stable bundle of degree $-1$ with low Segre invariants and ``large'' families of maximal subbundles for certain ranks. These $E$, which we construct in \S \ref{constructionofE}, are similar to examples of Ballico and Russo \cite{BR} of bundles whose Quot scheme $M_k(E)$ of maximal subbundles of rank $k$ is of large dimension. 

In \S \ref{sympexamples}, we adapt the construction to produce symplectic bundles $W$ of even rank $\geq 6$ with reducible theta divisors, and nonreduced if $r \geq 8$. These are obtained as extensions $0 \to E \to W \to E^* \to 0$ where $E$ is as above. To perform this construction, we obtain in \S \ref{sympliftings} some results on liftings in symplectic extensions which we hope may also be applicable in other contexts.\\
\\
\textbf{Acknowledgements:} I thank Insong Choe for helpful comments on this work, and Christian Pauly and Michel Raynaud for information on the construction in the appendix.

\section{Stable bundles with many maximal subbundles} \label{constructionofE}

In this section, we construct the bundles $E$ referred to in the introduction. 
We begin by recalling some results on vector bundle extensions.

\subsection{Extensions, lifting and geometry}

Let $E$ and $F$ be vector bundles over a curve, and let $0 \to E \to W \to F \to 0$ be a nontrivial extension. In this section we recall some results on liftings of elementary transformations of $F$ to $W$.


Let $V$ be a vector bundle with $h^1 (V) \neq 0$, and write $\pi$ for the projection $\PP V \to X$. By Serre duality and the projection formula and since $\pi_* O_{\PP V}(1) = V^*$, we have an identification
\[ H^1 (X, V) \xrightarrow{\sim} H^0 ( \PP V, \pi^* \Kx \otimes O_{\PP V}(1))^*. \]
By standard algebraic geometry, we obtain a map $\PP V \dashrightarrow \PP H^1 (V)$. See \cite[\S 2]{CH3} for more information and other descriptions of this map.

If $V = \Hom(F, E) = F^* \otimes E$ then we may consider the locus $\Delta_{F^* \otimes E}$ of rank one tensors, which has dimension $\rank F + \rank E - 1$.

\begin{lemma} Let $E$, $F$ and $W$ be as above. If an elementary transformation
\[ 0 \to \tilde{F} \to F \to \tau \to 0 \]
with $\deg \tau \leq k$ lifts to a subsheaf of $W$, then the class $\delta(W)$ of the extension belongs to $\Sec^k \left( \psi(\Delta_{F^* \otimes E}) \right)$. \label{etlift} \end{lemma}
\begin{proof} This is proven in \cite[Theorem 4.4 (i)]{CH1}. Note that in \cite{CH1} there are various assumptions on the degrees and genericity of $E$ and $F$, which need not be satisfied in the present applications. However, the function of these assumptions is to ensure that $\PP \Hom(F, E)$ is \emph{embedded} in $\PP H^1 (\Hom(F, E))$, which we do not require here. \end{proof}

\subsection{Construction of stable  bundles with many maximal subbundles}

Let $X$ be a curve of genus $g \geq 4$. Here we construct the bundles $E$ mentioned in the introduction. Like the bundles with large $M_k(E)$ constructed by Ballico and Russo \cite{BR}, these $E$ will be extensions of a decomposable bundle by a line bundle.

Choose a generic stable bundle $V \to X$ of degree zero and rank $n \geq 1$, and a positive integer $t$ with
\begin{equation} t < n(g-1). \label{tng} \end{equation}
Let $L \to X$ be a line bundle of degree $-1$, and consider a generic extension $0 \to L \to E \to V \otimes \C^t \to 0$. The following two lemmas form a partial analogue of the Claim in the proof of \cite[Theorem 0.0.1]{BR}:

\begin{lemma} Every subbundle of $E$ has negative degree. \label{SegreE} \end{lemma}
\begin{proof} Let $F$ be a proper subbundle of $E$. Then $F$ fits into a diagram
\[ \xymatrix{ 0 \ar[r] & L \ar[r] & E \ar[r] & V \otimes \C^t \ar[r] & 0 \\
0 \ar[r] & F_{1} \ar[r] \ar[u] & F \ar[r] \ar[u] & F_{2} \ar[r] \ar[u] & 0 } \]
where $F_1$ is either zero or $L$, and $F_2$ is a subsheaf of $V \otimes \C^t$. If $F_1 = L$ or if $t = 1$ then clearly $F$ has negative degree. If $t \geq 2$, we must show that no subbundle of the form $V \otimes \Lambda$ lifts to $E$, where $\Lambda \subset \C^t$ is a proper vector subspace. Clearly it suffices to treat the case $\dim \Lambda = 1$. We need to check that subspaces of the form
\begin{equation} \Ker \left( H^1 (\Hom(V \otimes \C^t , L)) \to H^1 (\Hom(V \otimes \Lambda, L)) \right) \label{kerlt} \end{equation}
do not sweep out $H^1 (\Hom(V \otimes \C^t , L))$. Since $V$ is stable, $h^{0}(\Hom(V, L)) = 0$. The dimension of (\ref{kerlt}) is therefore $(t-1) h^1 (\Hom(V, L))$.
 Furthermore, the subspaces $\Lambda$ vary in $\PP^{t-1}$. Therefore, it suffices to check that
\[ (t - 1) + (t-1)h^1 (\Hom(V, L)) < t \cdot h^1 (\Hom(V, L)), \]
that is, $t - 1 < h^1 (\Hom(V, L))$.

By Riemann--Roch and since $h^0(\Hom(V, L)) = 0$, we have $h^1(\Hom(V, L)) = ng$. 
The inequality $t-1 < ng$ follows from assumption (\ref{tng}), and we are done.
\end{proof}

\begin{lemma} \label{degminusone} Let $E$ be a generic extension of $V \otimes \C^t$ by $L$ as above. Then all degree $-1$ subbundles of $E$ contain the subbundle $L$. \end{lemma}
\begin{proof} 
We proceed by induction on $t$. It is convenient to begin with the case $t = 1$, although in applications we will most often assume that $t \geq 2$.

Consider an extension $0 \to L \to E \to V \to 0$. Any degree $-1$ subbundle $F$ of $E$ not containing $L$ must lift from a subsheaf of $V$.

\begin{proposition} A generic vector bundle $V$ of rank $n \geq 2$ and degree zero over a curve of genus $g \geq 4$ has no subbundles of degree $-1$. \label{vminusone} \end{proposition}
\begin{proof} By Russo--Teixidor i Bigas \cite[Theorem 0.2]{RT}, the Quot scheme of subsheaves of degree $-1$ and rank $m$ of a generic $V$ is empty when the expected dimension $n - m(n-m)(g-1)$ is negative. One checks easily that the maximum value of this dimension occurs at $m = 1$ and $m=n-1$, when it is equal to $n - (n-1)(g-1)$. Since $g \geq 4$, the required inequality would follow from $n - 3(n-1) < 0$, which is clear since $n \geq 2$.
\end{proof}

By the proposition, any subbundle $F \subset E$ of degree $-1$ not containing $L$ must lift from an elementary transformation $F \to V \to \C_x$. By Theorem \ref{etlift}, this happens only if the extension class of $E$ belongs to the image of the scroll $\PP \Hom(V, L)$ in $\PP H^1 (\Hom(V, L))$. Since $L$ is a line bundle, $\Delta_{V^* \otimes L} \cong \PP V^*$, which has dimension $n$. On the other hand, $h^1 (\Hom(V, L)) - 1 = ng - 1$. Since $g \geq 4$, a general extension class $\delta(E)$ does not belong to $\Delta_{V^* \otimes L}$. Thus we have proven the lemma for $t = 1$.

Now suppose $t \geq 2$, and let $E$ be a generic extension $0 \to L \to E \to V \otimes \C^t \to 0$. Choose a subspace $\Lambda \subset \C^t$ of dimension $t-1$, and consider the diagram
\[ \xymatrix{ & 0 \ar[d] & 0 \ar[d] & 0 \ar[d] & \\
 0 \ar[r] & L \ar[r] \ar[d] & E_0 \ar[r] \ar[d] & V \otimes \Lambda \ar[r] \ar[d] & 0 \\
 0 \ar[r] & L \ar[r] \ar[d] & E \ar[r] \ar[d] & V \otimes \C^t \ar[r] \ar[d] & 0 \\
 & 0 \ar[r] & V \ar[r] \ar[d] & V \ar[r] \ar[d] & 0 \\
 & & 0 & 0 & } \]
Since the induced map $H^{1}( \Hom(V \otimes \C^t, L)) \to H^{1} (\Hom(V \otimes \Lambda, L))$ is surjective, we may assume that $E_0$ is a generic extension of $V \otimes \Lambda$ by $L$.

Suppose $F \subset E$ is a proper subbundle of degree $-1$. Then we have a diagram
\[ \xymatrix{ 0 \ar[r] & E_0 \ar[r] & E \ar[r] & V \ar[r] & 0 \\
0 \ar[r] & F_1 \ar[r] \ar[u] & F \ar[r] \ar[u] & F_2 \ar[r] \ar[u] & 0 } \]
where $F_1$ is a subbundle of $E_0$ and $F_2$ a subsheaf of $V$.

Firstly, suppose $F_1 \neq 0$. By Lemma \ref{SegreE} we have $\deg F_1 \leq -1$, and therefore $\deg F_2 \geq 0$. Since $V$ is stable, the only possibilities are $F_2 = 0$ and $F_2 = V$, and so in fact $\deg F_1 = -1$. By induction, $L$ belongs to $F_1$ and hence to $F$.

On the other hand, if $F_1 = 0$ then $F \cong F_2$ must lift from a degree $-1$ subsheaf of $V$. By Proposition \ref{vminusone}, the only possibility is that $F_2$ is an elementary transformation
\begin{equation} 0 \to F \to V \to \tau \to 0 \label{etf} \end{equation}
where $\tau$ is a torsion sheaf of degree 1. By Theorem \ref{etlift}, the lifting of such an $F$ implies that the class $\varepsilon$ of the extension
\[ 0 \to E_0 \to E \to V \to 0 \]
belongs to the image of the scroll $\PP \Hom( V, E_0 )$ in $\PP H^1 (\Hom(V, E_0))$.

We note that since $E/L$ is a direct sum $V \otimes \C^t$, the class $\varepsilon$ belongs to
\[ \Ker \left( H^1 (\Hom(V, E_0)) \to H^1 (\Hom(V, V \otimes \Lambda)) \right), \]
which has dimension
\[ h^1 (\Hom(V, L)) - h^0 (\Hom(V, V \otimes \Lambda)) = ng - (t - 1). \]
Conversely, it is easy to see that any element of this kernel gives an extension $E$ of the form we began with.

We claim now that the intersection of $\psi \left(\Delta_{V^* \otimes E_0} \right)$ with
\[ \Image \left( \PP H^1 (\Hom(V, L)) \dashrightarrow \PP H^1 (\Hom(V, E_0)) \right) \]
is exactly $\psi \left( \PP \Hom(V, L) \right)$. Suppose $\psi (v^* \otimes e_0 ) \in \PP H^1 (\Hom (V, L))$ where $e_0 \not\in L$. Write $v_0$ for the image of $e_0$ in $V \otimes \Lambda$. Then the corresponding point $v^* \otimes v_0$ is a base point of the natural map
\[ \PP (\End(V) \otimes \Lambda ) \dashrightarrow \PP H^1 (\End(V) \otimes \Lambda ). \]
But it follows from Hwang--Ramanan \cite[Proposition 3.2]{HR} that this map is base point free (in fact an embedding) for general $V$. Thus if $\psi (v^* \otimes e_0) \in \PP H^1 (\Hom(V, L))$ then in fact $e_0 \in L$.

By the claim, we must check that
\[ \dim \PP \Hom(V, L) < h^1 (\Hom(V, L)) - h^0 (\Hom(V, V \otimes \Lambda)) - 1, \]
that is, $n < ng - (t-1) - 1$. This is exactly the assumption (\ref{tng}). 
Hence a generic extension $0 \to E_0 \to E \to V \to 0$ of our preferred type admits no lifting of the form (\ref{etf}), and we are done. \end{proof}

\begin{corollary} \label{mosubbsE} \begin{enumerate}
\renewcommand{\labelenumi}{(\arabic{enumi})}
\item Any subbundle $F \subseteq E$ of degree $-1$ is of the form $0 \to L \to F \to V \otimes \Lambda \to 0$, where $\Lambda$ is a uniquely determined vector subspace of $\C^t$.
\item The degree $-1$ subbundles of $E$ are parametrized by the union of the Grassmann varieties $\Gr(\C^t , s)$ for $s \in \{ 0, \ldots , t \}$.
\end{enumerate} \end{corollary}
\begin{proof} This is straightforward to check, in view of Lemma \ref{degminusone} and since $V$ is stable.
\end{proof}

%

\section{Stable bundles with reducible and nonreduced theta divisors} \label{mainexs}

\noindent We continue to assume that $X$ has genus $g \geq 4$.

\begin{proposition} Let $M$ be a generic stable bundle of rank $m \geq 1$ and degree $1$. Then any proper subbundle of $M$ has negative degree. \label{SegreM} \end{proposition}
\begin{proof} Similar to Proposition \ref{vminusone}. \end{proof}

\begin{theorem} A generic extension $0 \to E \to W \to M \to 0$ is a stable vector bundle. \label{Wstable} \end{theorem}

\begin{proof} Suppose $G \subset W$ is a proper subbundle. If $G$ is contained in the subbundle $E$ then $\deg G \leq -1$ by Lemma \ref{SegreE}. Otherwise, we have a diagram
\[ \xymatrix{ 0 \ar[r] & E \ar[r] & W \ar[r] & M \ar[r] & 0 \\
0 \ar[r] & F \ar[r] \ar[u] & G \ar[r] \ar[u] & H \ar[r] \ar[u] & 0 } \]
where $F$ is a subbundle of $E$ and $H$ a subsheaf of $M$. By Proposition \ref{SegreM}, it suffices to exclude liftings of the following types to $W$:
\begin{enumerate}
\renewcommand{\labelenumi}{(\roman{enumi})}
\item extensions $0 \to F \to G \to M \to 0$ where $F \subset E$ has degree $-1$; and 
\item degree zero elementary transformations of $M$.
\end{enumerate}

(i) Suppose $F \subset E$ is a proper subbundle of degree $-1$. By Corollary \ref{mosubbsE} (1), we have $\rank F = sn + 1$ for some $0 \leq s < t$. Then an extension $G$ of $M$ by $F$ belongs to $W$ if and only if $\delta(W)$ belongs to
\[ \Image \left( H^{1}(\Hom(M, F)) \to H^{1}(\Hom(M, E)) \right). \]
By Corollary \ref{mosubbsE} (2), 
to exclude case (i) in general, it will suffice to show that
\[ h^{1}(\Hom(M, E)) - h^{1}(\Hom(M, F)) - \dim \Gr(\C^t, s) > 0. \]
A straightforward calculation using Riemann--Roch shows that this would follow from $((m(g-1) + 1)n - s)(t - s) > 0$. 
Since $t > s$, we have $s < t < n(g-1)$ by (\ref{tng}), and then clearly $((m(g-1) + 1)n - s)(t - s) > 0$ as desired.

(ii) By Theorem \ref{etlift}, a degree 0 elementary transformation of $M$ lifts to $W$ only if the extension class of $W$ belongs to $\psi \left( \Delta_{M^* \otimes E} \right)$ in $\PP H^1 (\Hom(M, E))$. We have
\begin{multline*} \dim \psi \left( \Delta_{M^* \otimes E} \right) \ \leq \ m + tn \\ < \ m + tn + m(tn+1)(g-1) \ = \ \dim \PP H^1 (\Hom(M, E)), \end{multline*}
so a general extension $W$ admits no such lifting.
\end{proof}

Now we study theta divisors of such extensions $W$. Suppose $t \geq 2$, and let $V$, $L$, $E$ and $M$ be as above. Since $V$ is generic, we may assume $V$ has a reduced theta divisor $\Theta_V$.

\begin{theorem} \label{main} A generic extension $0 \to E \to W \to M \to 0$ has a reducible theta divisor $\Theta_W$ which scheme-theoretically contains $(t-1)\Theta_V$. In particular, if $t \geq 3$ then $\Theta_W$ is also nonreduced. \end{theorem}
\begin{proof}
For $P \in \Theta_V$, consider the exact sequence
\[ 0 \to P \otimes L \to P \otimes E \to (P \otimes V) \otimes \C^t \to 0. \]
Since $L$ and $V$ are generic, we may assume for generic $P \in \Theta_V$ that $h^0 (P \otimes L) = 0$. As $\Theta_V$ is reduced, furthermore $h^0 (P \otimes V) = 1$ for generic $P \in \Theta_V$ by Theorem \ref{Riemann}. 
Taking global sections, we obtain
\[ 0 \to H^0 (P \otimes E) \to H^0 (P \otimes V) \otimes \C^t \to H^1 (P \otimes L) \to \cdots \]
By Riemann--Roch, $h^1 (P \otimes L) = 1$. 
By exactness, $h^0 (P \otimes E) \geq t - 1$ for all $P \in \Theta_V$. Since $E \subset W$, we have $h^0 (P \otimes W) \geq t-1$ for all $P \in \Theta_V$. Thus by Theorem \ref{Riemann}, the theta divisor $\Theta_W$, if defined, has multiplicity at least $t-1$ at all $P \in \Theta_V$, and thus is of the form $(t-1)\Theta_V + R_W$ where $R_W$ is a divisor on $J$ algebraically equivalent to $(n+1+m)\Theta$.

Now we check that the theta divisor of a generic such $W$ is defined. Since $L$ is generic of degree $-1$, we have $h^0 (N \otimes L) = 0$ for all $N \in J$ outside a locus of codimension at least 2. Therefore, for generic $N \in J \backslash \Supp \, \Theta_V$, any map $N^{-1} \to W$ must lift from a map $N^{-1} \to M$. By Hirschowitz's Lemma \cite[Theorem 1.2]{RT} and Riemann--Roch, we have moreover $h^0 (N \otimes M) = 1$.

As $N^{-1}$ is a subsheaf of $M$, the induced map $H^1 (\Hom(M, E)) \to H^1 (\Hom( N^{-1}, E))$ is surjective. Since the latter space is nonzero, the map $N^{-1} \to M$ does not lift to a generic extension $0 \to E \to W \to M \to 0$. In particular, a generic such $W$ has a well-defined theta divisor.

As for reducibility: it is easy to find an extension $W$ satisfying $h^0 (P \otimes W) > 0$ for at least one $P$ with $h^0 (P \otimes E) = 0$. For such a $W$ the divisor $\Theta_W$ also has a component containing $P$, hence distinct from $\Theta_V$, and so is reducible. \end{proof}

In particular, we have:

\begin{corollary} \begin{enumerate}
\renewcommand{\labelenumi}{(\arabic{enumi})}
\item Let $X$ be a curve of genus $g \geq 5$. 
There exist stable bundles of rank $r$ over $X$ for all $r \geq 4$ (resp., $\geq 5$) with reducible (resp., reducible and nonreduced) theta divisors.
\item If $g = 4$ then there exist stable bundles of rank $r$ over $X$ for all $r \geq 4$ (resp., $r \geq 8$) with reducible (resp., reducible and nonreduced) theta divisors.
\end{enumerate} \end{corollary}
\begin{proof} If $g \geq 5$ then we may set $n = 1$ and $t = 2$ or $3$ and $m$ arbitrary. This gives (1).

As for (2): The only place above where we use the fact that $g \neq 4$ is to ensure that we can choose $t$ with $3 \leq t < n(g-1)$ when $n = 1$. Thus if $g = 4$, we can find stable bundles of rank $\geq 4$ with reducible theta divisors as in case (1).

If we set $n = 2$ and $t = 3$ and $m$ arbitrary, then the construction gives stable bundles of rank $\geq 8$ with reducible and nonreduced theta divisors. \end{proof}

\subsection{The residual divisor $R_W$} \label{resdiv} Consider a generic $W$ with extension class $\delta(W)$ and theta divisor $(t-1)\Theta_V + R_W$. As in \cite[\S 5]{H3}, we can give a geometric description of $R_W$ as follows:

Let $N$ be a generic line bundle of degree $g-1$ satisfying $h^0 (N \otimes E) = 0$ and $h^0 (N \otimes M) = 1$. Then by Riemann--Roch,
\[ \Image \left( m_N \colon H^0 (\Kx N^{-1} \otimes E^*) \otimes H^0 (N \otimes M) \to H^0 (\Kx \otimes E^* \otimes M) \right) \]
is of dimension 1. The association $N \mapsto \Image\,m_N$ defines a rational map
\[ \mu \colon J \dashrightarrow \PP H^0 (\Kx \otimes E^* \otimes M) \cong \PP H^1 (M^* \otimes E)^* . \]
Since $m_N$ is injective, $N$ belongs to the indeterminacy locus of $\mu$ if and only if $h^0 (N \otimes E) > 0$ or $h^0(N \otimes M) > 1$.

We write $R'_W$ for the complement of the indeterminacy locus in $R_W$, and $H_W$ for the hyperplane defined by $\delta(W)$ on $\PP H^0 (\Kx \otimes E^* \otimes M)$.

\begin{lemma} The set-theoretic intersection of $H_W$ and $\mu(J)$ is exactly $R'_W$. \label{resdivgeom} \end{lemma}
\begin{proof} Suppose $N \in J$ lies outside the indeterminacy locus of $\mu$. We have
\[ 0 \to H^0 (N \otimes W) \to H^0 (N \otimes M) \xrightarrow{\cdot \cup \, \delta(W)} H^1 (N \otimes E) \to \cdots \]
Then $N \in R_W$ if and only if the (up to scalar) unique section $\sigma$ of $N \otimes M$ satisfies $\sigma \cup \delta(W) = 0$. By hypothesis, $h^0 (N \otimes M) = 1 = h^1 (N \otimes E)$, so this is equivalent to
\[ \cdot \cup \delta(W) = 0 \ \in \ \Hom \left( H^0 (N \otimes M), H^1 (N \otimes E) \right). \]
Now it is well known that via Serre duality, the cup product map
\[ H^1 (\Hom(M, E)) \to \Hom \left( H^0 (N \otimes M), H^1 (N \otimes E) \right) \] 
is dual to $m_N$. Thus $h^0 (N \otimes W) > 0$ if and only if $m_N^* \delta(W) = 0$; in other words, $\mu(N) \in H_W$. 
The lemma follows. \end{proof}

\section{Symplectic extensions and liftings} \label{sympliftings}

Recall that a vector bundle $W$ is \textsl{symplectic} if there is an antisymmetric isomorphism $W \xrightarrow{\sim} W^*$; equivalently, if there exists a global bilinear nondegenerate antisymmetric form on $W$. In this section we gather some facts about such bundles.

\begin{criterion} \label{sympcrit} Let $E \to X$ be a simple vector bundle and $0 \to E \to W \to E^* \to 0$ an extension of class $\delta(W) \in H^1 (\Hom(E^*, E)) = H^1 (E \otimes E)$. Then $W$ carries a symplectic form with respect to which $E$ is isotropic if and only if $\delta(W)$ belongs to the subspace $H^1 (\Sym^2 E)$. \end{criterion}
\begin{proof} This is a special case of \cite[Criterion 2.1]{H1}. \end{proof}

\begin{proposition} \label{int} Let $E$ be any vector bundle and $F \subseteq E$ a subbundle. Then
\[ \left( F \otimes E \right) \cap \Sym^2 E = \Sym^2 F. \]
\end{proposition}
\begin{proof} The question is local. For some $x \in X$, suppose
\[ \sum e_i \otimes f_i \ \in \ (E \otimes F)|_x \cap \Sym^2 E|_x, \]
where each $e_i \in E|_x$ and $f_i \in F|_x$. Furthermore, we assume that the sum is of minimal length (equal to the rank of the associated map $E^*|_x \to F|_x$). Then
\[ \sum e_i \otimes f_i = \sum f_i \otimes e_i \]
and so by minimality $e_i = f_{\rho(i)}$ for some permutation $\rho$ of the indices. Hence all the $e_i$ belong to $F|_x$. The proposition follows.
\end{proof}
Now let $E$ be a vector bundle and $F \subset E$ a subbundle, and write $G := E/F$. If $h^0( G \otimes E) = 0$, then using Proposition \ref{int} we find a diagram
\begin{equation} \xymatrix{ & & 0 \ar[d] & H^0 (\wedge^2 G) \ar[d] & \\
0 \ar[r] & H^1 ( \Sym^2 F ) \ar[r] \ar[d] & H^1 ( \Sym^2 E ) \ar[r]^{b} \ar[d] & H^1 \left( \frac{\Sym^2 E}{\Sym^2 F} \right) \ar[r] \ar[d] & 0 \\
0 \ar[r] & H^1 ( F \otimes E ) \ar[r] \ar[d] & H^1 ( E \otimes E ) \ar[r] \ar[d] & H^1 ( G \otimes E ) \ar[r] \ar[d] & 0 \\
H^0 (\wedge^2 G) \ar[r] & H^1 \left( \frac{F \otimes E}{\Sym^2 F} \right) \ar[r] \ar[d] & H^1 ( \wedge^2 E ) \ar[r] \ar[d] & H^1 ( \wedge^2 G ) \ar[r] \ar[d] & 0 \\
 & 0 & 0 & 0 & } \label{cohomdiag} \end{equation}

\begin{lemma} \label{symplifting} Let $E$, $F$ and $G$ be as above. Then the subbundle $G^* \subseteq E^*$ lifts to $W$ if and only if $\delta(W)$ belongs to the preimage of $H^0 (\wedge^2 G)$ via the map $b$ in (\ref{cohomdiag}). 
In particular, if $h^0 (\wedge^2 G) = 0$ then $G^*$ lifts to $W$ if and only if $\delta(W) \in H^1 (\Sym^2 F)$. \end{lemma}
\begin{proof} It is well known that $G^*$ lifts to $W$ if and only if $\delta(W)$ belongs to
\[ \Ker \left( H^1 (E \otimes E) \to H^1 ( G \otimes E ) \right) = H^1 (F \otimes E). \]
Thus we must describe $H^1(F \otimes E) \cap H^1 (\Sym^2 E)$. In (\ref{cohomdiag}), we have
\[ H^1(F \otimes E) \cap H^1( \Sym^2 E ) = \Ker \left( H^1 (\Sym^2 E) \to H^1(E \otimes E) \to H^1 (G \otimes E) \right). \]
By commutativity, this coincides with
\[ \Ker \left( H^1 (\Sym^2 E) \xrightarrow{b} H^1 \left( \frac{\Sym^2 E}{\Sym^2 F} \right) \to H^1 ( G \otimes E) \right). \]
Thus $H^1 (F \otimes E) \cap H^1 (\Sym^2 E)$ is the preimage of $H^0 (\wedge^2 G)$ by $b$. The lemma follows. \end{proof}

\section{Symplectic bundles with reducible and nonreduced theta divisors} \label{sympexamples}

Here we adapt the construction of \S \ref{mainexs} to produce stable symplectic bundles with reducible and nonreduced theta divisors. As before, suppose $X$ has genus $g \geq 4$. Let $L$ be a line bundle of degree $-1$ and $V$ a generic stable bundle of rank $n \geq 1$ and degree zero. Let $E$ be a generic extension $0 \to L \to E \to V \otimes \C^t \to 0$, where $t < n(g-1)$.

\begin{lemma} A generic symplectic extension $0 \to E \to W \to E^* \to 0$ is a stable vector bundle. \label{sostability} \end{lemma}
\begin{proof} From Lemma \ref{SegreE} it follows that every nonzero quotient of $E^*$ has positive degree, and hence every proper subbundle of $E^*$ has nonpositive degree. Since $W$ is nonsplit, it is semistable. Suppose $F \subset W$ is a subbundle of degree zero. Then we have a diagram
\begin{equation} \xymatrix{ 0 \ar[r] & E \ar[r] & W \ar[r] & E^* \ar[r] & 0 \\
0 \ar[r] & F_1 \ar[r] \ar[u] & F \ar[r] \ar[u] & F_2 \ar[r] \ar[u] & 0 } \label{degzerosubbext} \end{equation}
where $F_1$ is a subbundle of $E$ and $F_2$ a subsheaf of $E^*$. We distinguish three cases:
\begin{enumerate}
\renewcommand{\labelenumi}{(\roman{enumi})}
\item If $F_1 \neq 0$ then $\deg F_1 \leq -1$ by Lemma \ref{SegreE}, whence $F_2 = E^*$ and $\deg F_1 = -1$. By Corollary \ref{mosubbsE} (1), the bundle $F_1$ is an extension $0 \to L \to F_1 \to V \otimes \Lambda \to 0$ where $\Lambda$ is a (possibly zero-dimensional) subspace of $\C^t$.
\item If $F_1 = 0$ and $\rank F_2 < \rank E$, then $F = F_2$ is a proper subbundle of degree zero. By dualizing the statement of Lemma \ref{degminusone}, we see that $F$ is of the form $V^* \otimes \Pi$ for a nonzero subspace $\Pi \subset (\C^t)^*$.
\item If $F_1 = 0$ and $\rank F_2 = \rank E$ then $F = F_2$ is an elementary transformation of $E^*$ along a torsion sheaf of length 1.
\end{enumerate}
\noindent We deal with each of these possibilities in turn:

(i) Here we obtain a diagram
\[ \xymatrix{ & 0 \ar[d] & 0 \ar[d] & 0 \ar[d] & \\ 
0 \ar[r] & F_1 \ar[r] \ar[d] & F \ar[r] \ar[d] & E^* \ar[r] \ar[d] & 0 \\
0 \ar[r] & E \ar[r] \ar[d] & W \ar[r] \ar[d] & E^* \ar[r] \ar[d] & 0 \\
0 \ar[r] & V \otimes (\C^t / \Lambda) \ar[r]^{\sim} \ar[d] & V \otimes (\C^t / \Lambda) \ar[r] \ar[d] & 0 & \\
 & 0 & 0 & & } \]
Dualizing this diagram, we see that $W^*$ contains a subbundle of the form $V^* \otimes (\C^t / \Lambda)^*$ lifting from $E^*$. But since $W$ is self-dual, this means that we are also in situation (ii). Thus it suffices to exclude possibility (ii) in general.
\par
(ii) Clearly it suffices to treat the case $\dim \Pi = 1$. We show that for any inclusion $j \colon V^* \otimes \Pi \to E^*$, we have
\[ \dim \Ker \left(j^* \colon H^1 (\Sym^2 E) \to H^1 (\Hom(V^* \otimes \Pi, E)) \right) + \dim \PP^{t-1} < h^1 (\Sym^2 E). \]
We write $F_{\Pi}$ for the subbundle of $E$ defined by the diagram
\[ \xymatrix{ 0 \ar[r] & L \ar[r] & E \ar[r] & V \otimes \C^t \ar[r] & 0 \\
0 \ar[r] & L \ar[r] \ar[u] & F_{\Pi} \ar[r] \ar[u] & V \otimes \Pi^{\perp} \ar[r] \ar[u] & 0 } \]
On the vector bundle level, clearly we have
\[ \Ker \left( j^* \colon \Hom(E^* , E) \to \Hom(V^* \otimes \Pi, E) \right) \cong \Hom(F_{\Pi}^*, E). \]

Hence by Proposition \ref{int}, we have $\Ker \left( j^* \right) \cap \Sym^2 E = \Sym^2 F_{\Pi}$. Furthermore, $h^0 (V \otimes E) = 0$ by Lemma \ref{SegreE}. Therefore, by Lemma \ref{symplifting}, to show that situation (ii) does not arise in general, it suffices to show that
\[ h^1 (\Sym^2 F_{\Pi}) + h^0 (\wedge^2 (V \otimes \Pi)) + t-1 < h^1 (\Sym^2 E). \]
Since $V$ is generic, $h^0 (\wedge^2 (V \otimes \C^t )) = 0$. Then the required inequality follows from a computation using Riemann--Roch and the fact that $h^0(\Sym^2 F_{\Pi}) = 0$.

(iii) By Lemma \ref{etlift}, it suffices to show that $\dim \psi \left( \Delta_{E \otimes E} \right) < h^1 (\Sym^2 E) - 1$. This follows from
\[ 2tn + 1 < (tn+2) + \frac{(tn+1)(tn+2)}{2}(g-1) - 1, \]
which is clear.

In summary, a general symplectic extension $0 \to E \to W \to E^* \to 0$ admits no subbundles of nonnegative degree.
\end{proof}

We now describe the theta divisor of a generic such $W$.

\begin{theorem} \label{mainsymp} A general symplectic extension $0 \to E \to W \to E^* \to 0$ has a reducible theta divisor $\Theta_W$ which scheme-theoretically contains $(t-1) \left( \Theta_V + \Theta_{V^*} \right)$. In particular, if $t \geq 3$ then $\Theta_W$ is nonreduced. \end{theorem}
\begin{proof}
Exactly as in the proof of Theorem \ref{main}, we see that the theta divisor of a generic such $W$, if it exists, scheme-theoretically contains $(t-1)\Theta_V$. By Serre duality (see also Beauville \cite[\S 2]{B2}) we have $\iota^* \Theta_W = \Theta_W$, where $\iota$ is the involution of $J$ induced by $N \mapsto \Kx N^{-1}$. 
In general, $\iota^* \Theta_V = \Theta_{V^*}$ by Serre duality. Since $V$ is generic of degree zero, we may assume that $\Theta_V \neq \Theta_{V^*}$. Then $\Theta_W$ scheme-theoretically contains $(t-1) \left(\Theta_V + \Theta_{V^*} \right)$.

We check that $\Theta_W$ is defined. Choose a generic line bundle $N \in J \backslash \Supp \left( \Theta_V + \Theta_{V^*} \right)$. By Riemann--Roch and genericity, $h^0 (N \otimes E^* ) = 1$, and the corresponding map $j \colon N^{-1} \to E^*$ is a vector bundle injection. Dualizing, we obtain an exact sequence $0 \to F \to E \to N \to 0$, where $F := \left( E^* / N^{-1} \right)^*$. Now we claim that
\begin{equation} h^1 ( \Sym^2 E) - h^1 \left( \Sym^2 F \right) > 0. \label{FEineq} \end{equation}
Since $\Sym^2 F$ is a subbundle of $\Sym^2 E$, it has no global sections. Then a computation with Riemann--Roch shows that the left hand side of (\ref{FEineq}) is equal to $1$ (the expected value).

By Lemma \ref{symplifting} (with $G = N$), we have $h^0 (N \otimes W) = 0$ for a generic symplectic extension $0 \to E \to W \to E^* \to 0$. Hence such a $W$ has a well-defined theta divisor $\Theta_W$. As in Theorem \ref{main}, we check that $\Theta_W$ in general has at least one component distinct from $\Theta_V$. This completes the proof of Theorem \ref{mainsymp}. \end{proof}

\noindent We summarize as follows:
\begin{corollary}
\begin{enumerate}
\renewcommand{\labelenumi}{(\arabic{enumi})}
\item Let $X$ be a curve of genus $g \geq 5$. Then there exist stable symplectic bundles of rank $2k$ for all $k \geq 3$ (resp., $k \geq 4$) with reducible (resp., reducible and nonreduced) theta divisors.
\item If $g = 4$ then there exist stable symplectic bundles of rank $2k$ for all $k \geq 3$ (resp., $k \geq 7$) with reducible (resp., reducible and nonreduced) theta divisors.
\end{enumerate} \end{corollary}
\begin{proof} If $g \geq 5$ then we may as before set $n = 1$ and $t = 2$ or $3$. This proves (1).

As for (2): Setting $n=1$ and $t = 2$, we obtain stable symplectic bundles of rank $6$ with reducible theta divisors as in case (1). If we set $n = 2$ and $t = 3$, we obtain stable symplectic bundles of rank $14$ with reducible and nonreduced theta divisors. \end{proof}

\subsection{The residual divisor} As in \S \ref{resdiv}, we may give a geometric description of the residual divisor $R_W$ defined by $\Theta_W = (t-1)(\Theta_V + \Theta_{V^*}) + R_W$. Replacing $M$ with $E^*$, we define a map
\[ \mu_s \colon J \dashrightarrow \PP H^0 ( \Kx \otimes \Sym^2 E^* ) \]
by composing the projection $\PP H^0 ( \Kx \otimes E^* \otimes E^* ) \dashrightarrow \PP H^0 ( \Kx \otimes \Sym^2 E^*)$ with $\mu$. Note that if $h^0 (N \otimes E^*) = 1 = h^0 (\Kx N^{-1} \otimes E^*)$, then the image of
\[ m_N \colon H^0 (N \otimes E^*) \otimes H^0 (\Kx N^{-1} \otimes E^*) \ \to \ H^0 (\Kx \otimes E^* \otimes E^* ) \]
intersects $H^0 ( \Kx \otimes \wedge^2 E^* )$ in zero, 
 so the base locus of $\mu_s$ is no bigger than that of \nolinebreak $\mu$.

As before, an extension $0 \to E \to W \to E^* \to 0$ defines a hyperplane $H_W$ on $H^0 (\Kx \otimes E^* \otimes E^*)$. Since $\delta(W)$ belongs to $H^1 (\Sym^2 E)$, this $H_W$ cuts out a hyperplane $H_W^s$ on $\PP H^0 (\Kx \otimes \Sym^2 E^*)$.

As before, we write $R'_W$ for the complement of the base locus of $\mu_s$ in $R_W$.

\begin{lemma} The set-theoretic intersection of $\mu_s(J)$ and $H_W^s$ is exactly $R'_W$. \end{lemma}
\begin{proof} If $h^0 (N \otimes E^*) \cdot h^0 (\Kx N^{-1} \otimes E^*) = 1$, then exactly as in Lemma \ref{resdivgeom}, we see that $\mu (N) \in H_W$ if and only if $h^0 (N \otimes W) > 0$. Write
\[ S \colon H^0 ( \Kx \otimes E^* \otimes E^*) \to H^0 ( \Kx \otimes \Sym^2 E^* ) \]
for the projection. Since $\delta(W)$ is symmetric, $\delta(W) \circ S = \delta(W)$. Thus $\Image \, m_N \in \Ker \, \delta(W)$ if and only if $\Image (m_N \circ S) \subseteq H_W^s$; in other words, $\mu_s(N) \in H_W^s$. The lemma follows. \end{proof}

%

\begin{remark} Since $\Theta_W$ and $\Theta_V + \Theta_{V^*}$ belong to $|2(tn+1)\Theta|_+$ and $|2n\Theta|_+$ respectively, clearly $R_W$ belongs to the subspace $|2(n+1)\Theta|_+$ of $|2(n+1)\Theta|$. In fact $\mu_s$ is $\iota$-invariant by construction, and factorizes via the Kummer variety.

\end{remark}

\appendix

\section{Raynaud's example}

In this appendix we construct a stable rank $2$ vector bundle over a bi-elliptic curve of genus $g \geq 3$ having a 
reducible theta divisor. This construction is attributed to M. Raynaud.

\bigskip

Let $X$ be a smooth projective curve of genus $g \geq 3$. We assume that $X$ is bi-elliptic, i.e., $X$  admits a degree $2$ map to an elliptic 
curve $Z$
$$\pi \colon X \to Z.$$
We continue to write $J$ for the Picard variety of line bundles of degree $g-1$ over $X$,  $H^0(V)$ for $H^0 ( X, V )$
and $h^0(V)$ for $\dim  H^0(X,V)$.

Let $i$ be the sheet involution of $X$. The canonical bundle $\Kx$ of $X$ equals $\Ox(R)$, where $R$ is the ramification divisor of $\pi$. 
Let $M$ be a line bundle of degree $g$ on $Z$, and denote $L$ the degree $2$ line bundle over $X$ defined by $\Kx L = \pi^* M$.
We choose a square root $\zeta$ of $L$, so $\zeta^2 = L$, and $\deg \zeta = 1$. By Riemann--Roch, we have
\[ h^0 (Z, M) = g \quad \hbox{and} \quad h^0 (X, \Kx L) = g+1. \]
Hence the image of the injective map
\[ \pi^* \colon H^0(Z, M) \hookrightarrow H^0(X, \Kx L) \]
determines a point $e \in |\Kx L|^* \cong \PP^g$. We obtain a commutative diagram
\[ \xymatrix{X \ar[r] \ar[d]_{\pi} & |\Kx L|^* \cong \PP^g \ar[d]^{p} \\
 Z \ar[r] & |M|^* \cong \PP^{g-1} }. \]
The map $p = \PP(\pi^*)^*$  is projection with centre $e$. Moreover, the involution $i$ induces a decomposition into eigenspaces 
$H^0(\pi^* M)_- \oplus H^0(\pi^* M)_+$, where the second summand corresponds to $e$. We 
consider the rank $2$ bundle $E$ with trivial determinant given by the extension class $e \in |\Kx L|^* = \PP \mathrm{Ext}^1(\zeta, \zeta^{-1})$.
Hence $E$ fits into the  exact sequence
\begin{equation} 0 \to \zeta^{-1} \to E \to \zeta \to 0. \label{extE} \end{equation}

\begin{proposition} The bundle $E$ is stable. \end{proposition}
\begin{proof} The bundle $E$ is clearly semistable, since $e$ is nonzero. To see that is stable, it suffices by Lange--Narasimhan \cite[Proposition 1.1]{LNa} to show that $e$ does not belong to the image of $X$ in $|\Kx L|^*$. Now the involution $i$ induces a decomposition into eigenspaces of $H^0( \pi^* M)_- \oplus H^0(\pi^* M)_+$, where the second summand defines the hyperplane $e \in H^{0}(\Kx \zeta^2 )^*$. If $H^0(\pi^* M)_+$ is of the form $H^0 (\Kx L(-x))$ for some $x \in X$ then the linear system $|M|$ on $Z$ would have the base point $\pi(x)$. But this is impossible since $\deg M > 1$.
\end{proof}

Given $z \in Z$, we denote by $\overline{z}$ the effective degree $2$ divisor $\pi^{-1}(z)$. Note that for any $z \in Z$, the line spanned by $\overline{z}$ passes through $e$. By Lange--Narasimhan \cite[Proposition 5.3]{LNa}, the set of line subbundles of maximal degree $-1$ of $E$ contains
\[ \{ \zeta(-\overline{z}) \in \Pic^{-1}(X) : z \in Z \} \cong Z. \]
Let $\Theta_E \subset J$ be the theta divisor associated to $E$. 
We denote by $X^d$ and $X^{(d)}$ the $d$-th Cartesian and the $d$-th symmetric product of the curve $X$ respectively and by 
$q : X^d \to X^{(d)}$ the natural projection.
\begin{lemma} \label{phibir} 
The natural map $\phi \colon Z \times  X^{(g-2)} \to J$ defined by
\[ (z, D) \mapsto \Ox(\overline{z} + D) \otimes \zeta^{-1} \]
is birational onto its image, which is contained in the theta divisor $\Theta_E$. Therefore the divisor $\cD_1 := \mathrm{Im} \  \phi$ is an irreducible component of $\Theta_E$. 
\end{lemma}

\begin{proof} Since we have injections $\zeta(-\overline{z}-D) \hookrightarrow \zeta(-\overline{z}) \hookrightarrow E$ for any pair $(z, D)$, we obtain an inclusion $\cD_1 = \phi(Z \times X^{(g-2)}) \subset \Theta_E$. Moreover, if the pair $(z, D)$ is general and $\phi(z, D) = \phi(z', D')$, 
then $z = z'$ and $D = D'$. In fact, the subvariety of line bundles $N \in \Pic^g(X)$ with $h^0(N) \geq 2$ is of dimension $g-2$. Since
$\dim Z \times  X^{(g-2)} = g-1$, we obtain that $h^0(\Ox(\overline{z} + D)) =1$ for a general $(z,D) \in Z \times  X^{(g-2)}$.
\end{proof}

For any integer $d$ the Norm map of the covering $\pi : X \to Z$ induces a morphism between the Picard varieties of 
degree $d$ line bundles 
$$\Nm : \Pic^d (X) \to \Pic^d(Z).$$
We recall the formula $N \otimes i^* N = \pi^* \Nm(N)$ for any line bundle $N \in \Pic(X)$.  We write $P(X/Z)$ for 
the Prym variety of the cover $\pi \colon X \to Z$, i.e.,
$$ P(X/Z) = \Ker \ \Nm = \Ker \left( 1 + i^* \right) = \{ \eta \in \Pic^0(X) : i^* \eta = \eta^{-1} \}. $$
Finally for any $\lambda \in \Pic(Z)$ the fiber $\Nm^{-1}(\lambda) \subset \Pic(X)$ is a translate of the Prym variety $P(X/Z)$.
Note that $\dim P(X/Z) = g-1$ and that $\Ker \ \Nm$ is connected since $\pi$ is ramified.

\begin{lemma} 
We put $\lambda = M \otimes \Nm(\zeta)^{-1} \in \Pic^{g-1}(Z)$ and $\cD_2 = \Nm^{-1}(\lambda) \subset J$.
Then $\cD_2$ is an irreducible component of the theta divisor $\Theta_E$. Moreover $\cD_2 \not= \cD_1$.
\end{lemma}

\begin{proof}
We first show that $\cD_2 \subset \Theta_E$. The two subvarieties of line bundles $\eta \in \Pic^{g-1}(X)$ satisfying 
$h^0(\zeta^{-1} \eta) > 0$ and $h^0(\zeta \eta ) > 1$ respectively are of dimension $g-2$, so  $h^0(\zeta^{-1} \eta) = 0$ and
$h^0(\zeta \eta ) = 1$ for a general line bundle $\eta \in \cD_2$. Next, we observe that $\eta \in \cD_2$ if and only
if 
$$ \eta \otimes i^* \eta = \pi^* \Nm(\eta) =  \pi^* M  \otimes \zeta^{-1} \otimes  i^* \zeta^{-1} = \Kx \otimes \zeta  \otimes i^* \zeta^{-1},$$
or equivalently
$$i^*(\zeta \eta) = \Kx \zeta \eta^{-1}.$$
We now tensor the exact sequence \eqref{extE} with $\eta$ and take the associated long exact sequence of
cohomology 
$$ 0 \lra H^0(\zeta^{-1} \eta) \lra H^0(E \eta) \lra H^0(\zeta \eta) \map{\cup e} H^1(\zeta^{-1} \eta) \lra \cdots $$
Under the generality assumption for $\eta$ we have $h^0(\zeta^{-1} \eta) = 0$ and
$h^0(\zeta \eta ) = 1$, so $h^0(E \eta) > 0$ if and only if the coboundary map $\cup e$ is zero, or equivalently,
if the image of the multiplication map 
$$ H^{0}(\zeta \eta) \otimes H^0(\Kx \zeta \eta^{-1}) \lra H^0 ( \Kx \zeta^2 ) = H^0(\pi^* M) $$
lies in the invariant part $H^0(\pi^* M)_+$ defining the extension class $e$. We see that this is the case if 
$i^*(\zeta \eta) = \Kx \zeta \eta^{-1}$ and $h^0(\zeta \eta) =1$. Hence for a general $\eta \in \cD_2$ we have $h^0(E \eta) >0$,
which implies that $\cD_2 \subset \Theta_E$.

\bigskip

It is clear that the restriction of the Norm map to $\cD_1$ is not constant, hence $\cD_2 \not= \cD_1$.
\end{proof}

In fact, there are no other components.  This is shown in the following 

\begin{proposition}
We have a decomposition into irreducible components
$$ \Theta_E  = \cD_1 + \cD_2.$$
\end{proposition}

\begin{proof}
Let $\Theta \subset J$  denote the Riemann theta divisor. In order to show the equality it will be enough to show the 
equality of intersection numbers 

\begin{equation} \label{intnumb}
\cD_1 . \Theta^{g-1} + \cD_2 . \Theta^{g-1} = \Theta_E . \Theta^{g-1} = 2 \Theta^{g} = 2 g!.
\end{equation}

First we compute $\cD_1 . \Theta^{g-1}$. Let $\mathcal{L} = \phi^*(\mathcal{O}_J(\Theta))$ denote the pull-back of the line bundle
$\mathcal{O}_J(\Theta)$ to $Z \times X^{(g-2)}$. Since $\phi$ is birational by Lemma \ref{phibir}, we have $\cD_1 . \Theta^{g-1} =
 \mathcal{L}^{g-1}$.
We will use the following commutative diagram
$$ \xymatrix{ X \times X^{g-2} \ar[rr]^-{(\Id \times i) \times \Id} \ar[d]^{\pi \times q} & & (X \times X) \times X^{g-2} = X^g \ar[d]^{\alpha} \\
 Z \times X^{(g-2)} \ar[rr]^-{\phi} & & J } $$
where $\alpha$ is given by $\alpha((x_i)) = \Ox(\sum_{i=1}^g x_i) \otimes \zeta^{-1}$.  Hence, since $\deg (\pi \times q) = 2 (g-2) !$, we obtain
$$ \mathcal{L}^{g-1} = \frac{1}{2 (g-2) !} \left[ (\pi \times q)^* \mathcal{L}  \right]^{g-1} = 
\frac{1}{2 (g-2) !} \left[ (\Id \times i)^* \circ \alpha^*  \mathcal{O}_J(\Theta)   \right]^{g-1}.$$
We will compute the latter intersection number in the cohomology ring $H^*(X \times X^{g-2}, \C)$. A straightforward
computation leads to
\begin{equation} \label{pullbackTheta}
\alpha^* \mathcal{O}_J(\Theta) = \Kx \zeta \boxtimes \Kx \zeta \boxtimes \cdots \boxtimes \Kx \zeta (- \sum_{1 \leq i< j \leq g} \Delta_{ij} ) 
\end{equation}
where $\Delta_{ij} \subset X^g$ is the diagonal on the $i$-th and $j$-th component in $X^g$.

We need to recall some results from \cite{MD}. We denote by $\beta$ the generator of $H^2(X, \zz) \cong \zz$ induced by the orientation
of $X$ and we choose generators $\alpha_1, \ldots, \alpha_{2g}$ of $H^1(X, \zz) \cong \zz^{2g}$ such that 
$\alpha_j \alpha_k = 0$ unless $j-k = \pm g$, $\alpha_j \alpha_{g+j} = - \alpha_{g+j} \alpha_j = \beta$ for $1 \leq j \leq g$, and 
such that the involution $i$ acts as $i(\alpha_1) = \alpha_1$, $i(\alpha_{g+1}) = \alpha_{g+1}$ and $i(\alpha_j) = - \alpha_j$,
$i(\alpha_{g+j}) = - \alpha_{g+j}$ for $2 \leq j \leq g$. We also introduce for $1 \leq k \leq g-2$
$$ \alpha_{j,k} = 1 \otimes \cdots \otimes 1 \otimes \alpha_j \otimes 1 \otimes  \cdots \otimes 1  \qquad
\text{and} \qquad \beta_k = 1 \otimes \cdots \otimes 1 \otimes \beta \otimes 1 \otimes \cdots \otimes 1, $$
the $\alpha_j$ and $\beta$ being in the $k$-th place, as well as for $1 \leq j \leq 2g$
$$ \xi_j = \alpha_{j,1} + \cdots + \alpha_{j,g-2} \in H^1(X^{g-2}, \C) \qquad \text{and} \qquad \eta = \beta_1 + \cdots + \beta_{g-2}
\in H^2(X^{g-2}, \C).$$
We also put $\sigma_i = \xi_i \xi_{g+i}$ for $1 \leq i \leq g$ and we recall that we have (see \cite{MD}) the following
relations $\sigma_i^2 = 0$, $\sigma_i \sigma_j = \eta^2$ for $i \not= j$, $\sigma_i \eta = \eta^2$, and $\xi_i \sigma_j = \sigma_j
\xi_i$ for any $i,j$.

\bigskip

We introduce the ``diagonal" divisors in $X^{g-2}$ and in $X \times X^{g-2}$ with their reduced structures
$$  \Delta = \{ (x_1, \ldots , x_{g-2}) \in X^{g-2} \ | \ x_j = x_k \ \text{for some} \ 1 \leq j,k \leq g-2 \}, $$
$$  \Delta'_k = \{ (u; x_1, \ldots , x_{g-2}) \in X \times X^{g-2} \ | \ u =  x_k  \} \ \text{and} \ \Delta' = \sum_{k=1}^{g-2} \Delta'_k.$$
We denote by $p$ and $q$ the projection of $X \times X^{g-2}$ onto the first and second factor respectively. Then it follows
from \eqref{pullbackTheta} that 
$$(\Id \times i)^* \circ \alpha^*  \mathcal{O}_J(\Theta) = p^* \left( K_X \zeta \otimes i^* \zeta   \right) 
\otimes q^* \left( (K_X \zeta)^{\boxtimes {g-2}} (-\Delta)   \right) (-\Delta' - i(\Delta')).$$
We now compute the class $c$ of this line bundle in $H^*(X \times X^{g-2}, \C) = H^*(X, \C) \otimes H^*(X^{g-2}, \C)$. By \cite{MD} formula (15.4)
the class $[D] \in H^*(X, \C) \otimes H^*(X, \C)$ of the diagonal $D \subset X \times X$ equals
\begin{eqnarray*}
[D] & = & (g+1) (\beta \otimes 1 + 1 \otimes \beta) - \sum_{j=1}^g (\alpha_j \otimes 1 + 1 \otimes \alpha_j) 
(\alpha_{g+j} \otimes 1 + 1 \otimes \alpha_{g+j})\\
 & = & (\beta \otimes 1 + 1 \otimes \beta)  - \sum_{j=1}^g (\alpha_{g+j} \otimes  \alpha_j + \alpha_j \otimes \alpha_{g+j}).
\end{eqnarray*}
Let $\pi_k  : X^{g-2} \rightarrow X$ denote projection onto the $k$-th factor. Then $\Delta'_k = (\mathrm{Id} \times \pi_k)^{-1}(D)$
and therefore its class $[\Delta'_k] \in H^*(X,\C) \otimes H^*(X^{g-2}, \C)$ equals 
$$ [\Delta'_k] = \beta \otimes 1 + 1 \otimes \beta_k - \sum_{j=1}^g (\alpha_{g+j}  \otimes \alpha_{j,k} + \alpha_j \otimes \alpha_{g+j,k} ) $$
Summing over $k = 1, \ldots , g-2$ leads to 
$$[\Delta'] = (g-2) \beta \otimes 1 + 1 \otimes \eta - \sum_{j=1}^g (\alpha_{g+j}  \otimes \xi_j + \alpha_j \otimes \xi_{g+j} ) $$
and, applying the involution $i$ on the first factor $X$
$$[i(\Delta')] = (g-2) \beta \otimes 1 + 1 \otimes \eta - \alpha_{g+1}  \otimes \xi_1 - \alpha_1 \otimes \xi_{g+1} + \sum_{j=2}^g (\alpha_{g+j}  \otimes \xi_j + \alpha_j \otimes \xi_{g+j} ). $$
Moreover, again by \cite{MD} formula (15.4), we have
$$ [\Delta] =  (2g -3) \eta  - \sum_{j=1}^{g-2} \sigma_j \in H^*(X^{g-2}, \C),$$
$[K_X \zeta \otimes i^* \zeta ] = (2g) \beta \in H^2(X, \C)$, and  $[(K_X \zeta)^{\boxtimes {g-2}}] = (2g-1) \eta 
\in H^2(X^{g-2}, \C)$. Hence, using  the preceding equalities, we can 
compute the class
$$ c =  4 (\beta \otimes 1) + 2 ( \alpha_{g+1}  \otimes \xi_1 +  \alpha_1 \otimes \xi_{g+1} ) + \sum_{j=1}^g 1 \otimes \sigma_j. $$ 
We put $a = 4 (\beta \otimes 1) + 2 ( \alpha_{g+1}  \otimes \xi_1 +  \alpha_1 \otimes \xi_{g+1} )$ and 
$b=  \sum_{j=1}^g 1 \otimes \sigma_j$ and we note that $ab = ba$. Moreover, for dimensional reasons, $b^{g-1} = 0$ and $a^n = 0$ for
$n \geq 3$. Hence $c^{g-1} = (a+b)^{g-1} = (g-1) ab^{g-2} + \frac{(g-1)(g-2)}{2} a^2 b^{g-3}$. Using the above relations satisfied by the
$\sigma_i$, we easily compute that $ab^{g-2} = 4 \beta \otimes \left(  \sum_{j=1}^g  \sigma_j \right)^{g-2} = 2 g! \beta \otimes
\eta^{g-2}$, $a^2 = -8 \beta \otimes \sigma_1$ and $a^2 b^{g-3} = -4 (g-1) ! \beta \otimes \eta^{g-2}$. Using the fact that
$\beta \otimes \eta^{g-2} = (g-2) !$ under the canonical isomorphism $H^{2g-2}(X^{g-1}, \zz) = \zz$, we obtain that $c^{g-1} =
4 \left( (g-1)! \right)^2$. Hence $\cD_1 . \Theta^{g-1} = \frac{c^{g-1}}{2 (g-2) !} = 2 (g-1) (g-1)!$.

\bigskip

We now compute $\cD_2 . \Theta^{g-1}$. Let $L$
denote the restriction of the line bundle $\mathcal{O}_J(\Theta)$ to $\cD_2$. We recall that $\cD_2$ is a translate
of the Prym variety $P(Y/X)$. Then by 
\cite{BL} Corollary 12.1.5 the type of the polarization given by $L$ is $(1,1, \ldots, 1,2)$, hence
$h^0(\cD_2, L) = 2 = \frac{L^{g-1}}{(g-1)!}$ by the Riemann-Roch theorem applied to the Prym variety $P(X/Z) \cong \cD_2$. Therefore 
$\cD_2 . \Theta^{g-1} = L^{g-1} = 2 (g-1)!$.

We then conclude because we obtain equality \eqref{intnumb} by summing both intersection numbers. 
\end{proof}

\end{document}